\documentclass{amsart}
\usepackage{amsfonts,amssymb,amscd,amsmath,enumerate,verbatim,calc}
\usepackage[all]{xy}

\newcommand{\CM}{Cohen-Macaulay}

\newcommand{\n}{\mathfrak{n} }
\newcommand{\m}{\mathfrak{m} }
\newcommand{\M}{\mathfrak{M} }

\newcommand{\R}{\mathcal{R} }
\newcommand{\T}{\mathcal{T} }

\newcommand{\Lc}{\mathcal{L} }
\newcommand{\Sc}{\mathcal{S} }
\newcommand{\Fc}{\mathcal{F} }
\newcommand{\Z}{\mathbb{Z} }

\newcommand{\rt}{\rightarrow}

\newcommand{\ov}{\overline}

\newcommand{\grade}{\operatorname{grade}}

\theoremstyle{plain}

\newtheorem{theorem}{Theorem}[section]

\newtheorem{lemma}[theorem]{Lemma}
\newtheorem{proposition}[theorem]{Proposition}

\theoremstyle{definition}

\newtheorem{s}[theorem]{}

\theoremstyle{remark}

\begin{document}

\title[local cohomology]{Integral closure and local cohomology}
\author{Tony~J.~Puthenpurakal}
\date{\today}
\address{Department of Mathematics, IIT Bombay, Powai, Mumbai 400 076}
\email{tputhen@math.iitb.ac.in}

 \begin{abstract}
Let $A$ be a Noetherian ring and let $I$ be an ideal in $A$. Let $\Fc = \{ J_n \}_{n \geq 0}$ be a multiplicative filtration of ideals in $A$ such that $\R(\Fc) = \bigoplus_{n \geq 0} J_n$ is a finitely generated $A$-algebra. Let $\R = A[It]$ and assume  $I^n \subseteq J_n$ for all $n \geq 1$. We show the following two assertions are equivalent:
\begin{enumerate}
  \item For all $i \geq 0$ we have $H^i_{\R_+}(\R(\Fc))_n = 0$ for all $n \gg 0$.
  \item $J_n \subseteq \ov{I^n}$ for all $n \geq 1$.
\end{enumerate}
Here $\ov{I^n}$ is the integral closure of $I^n$.
\end{abstract}
 \maketitle
\section{introduction}
Let $A$ be a Noetherian ring.  If $K$ is an ideal of $A$ then $\ov{K}$ denotes the integral closure of $K$.
In this paper we relate local cohomology and integral closure. We prove the following result:
\begin{theorem}\label{main}
Let $A$ be a Noetherian ring and let $I$ be an ideal in $A$. Let $\Fc = \{ J_n \}_{n \geq 0}$ be a multiplicative filtration of ideals in $A$ such that $\R(\Fc) = \bigoplus_{n \geq 0} J_n$ is a finitely generated $A$-algebra. Let $\R = A[It]$ and assume  $I^n \subseteq J_n$ for all $n \geq 1$. The following two assertions are equivalent:
\begin{enumerate}[\rm (1)]
  \item For all $i \geq 0$ we have $H^i_{\R_+}(\R(\Fc))_n = 0$ for all $n \gg 0$.
  \item $J_n \subseteq \ov{I^n}$ for all $n \geq 1$.
\end{enumerate}
\end{theorem}
Note the assertion (2) $\implies$ (1) is trivial and is left to the reader. The real content of the result is (1) $\implies$ (2).

The essential ingredient is the following result:
\begin{theorem}\label{ingredient}
Let $(A,\m)$ be a Noetherian local ring and $I \subseteq J \subseteq \m$ with $\ell(J/I) < \infty$. Set $\R = A[It]$ and $\Sc = A[Jt]$. Suppose for all $i \geq 0$ we have $H^i_{\R_+}(\Sc)_n = 0$ for all $n \gg 0$. Then $J \subseteq \ov{I}$.
\end{theorem}

The proof of Theorem \ref{main} follows from Theorem \ref{ingredient} by taking an appropriate Veronese of $\R(\Fc)$ and an appropriate localization. However we need to ensure that $JA_P$ does not blowup to $A_P$ while $IA_P$ remains $P$-primary. So we need the following result:
\begin{lemma}\label{blow-up}
Let $(A,\m)$ be a Noetherian local ring of dimension $d \geq 0$ and $I$ an  $\m$-primary ideal. Set $\R = A[It]$ and $\Sc = A[t]$.  Then there exists $i \geq 0$ such that $H^i_{\R_+}(\Sc)_n \neq 0$ for infinitely many $n  \geq 0$.
\end{lemma}
Here is an overview of the contents of this paper. In section two we discuss a few preliminary results that we need. In section three we prove Theorem \ref{ingredient}. In section four we prove
Lemma \ref{blow-up}. Finally in section five we prove Theorem \ref{main}.
\section{Preliminaries}
In this section we discuss a few preliminary results that we need. In this paper all rings considered are Noetherian.

Let $(A,\m)$ be a local ring of dimension $d$ and let $I$ be an ideal in $A$. Let $\R = A[It]$.  Let $J \supseteq I$ be another ideal and set $\Sc = A[Jt]$. Note $\Sc$ is an $\R$-module.
\s\label{filt} Let $M = \bigoplus_{n \geq 0}M_n$ be an $\R$-module. We say $M$ is a \emph{quasi-finite} $\R$-module if
$M_n $ is a finitely generated $A$-module for all $n \geq 0$ and $H^0_{\R_+}(M)_n = 0$ for $n \gg 0$. If the residue field of $A$ is uncountable and $M$ is a quasi-finite $\R$-module then there
exists $xt \in \R_1$ which is $M$-filter regular, i.e., $(0\colon_M xt)_n = 0$ for $n \gg 0$, see \cite[2.7]{HPV}.

\s Set $\Lc^J = \bigoplus_{n \geq 0}A/J^{n+1}$. Note $\Sc$ is a $\R$-submodule of $A[t]$ and $A[t]/\Sc = \Lc^J(-1)$. Thus $\Lc^J$ is an $\R$-module.
This module was introduced in \cite{P1}.

\s Let $G_J(A) = \bigoplus_{n\geq 0}J^n/J^{n+1}$ be the associated graded ring of $J$. Then $G_J(A)$ is a quotient ring of $\Sc$ and so an $\R$-module.
The exact sequence
$$ 0 \rt \frac{J^n}{J^{n+1}} \rt \frac{A}{J^{n+1}} \rt \frac{A}{J^n} \rt 0 \quad \text{for $n \geq 0$}, $$
induces an exact sequence of $\R$-modules
$$ 0 \rt G_J(A) \rt \Lc^J \rt \Lc^J(-1) \rt 0.$$

\s Assume $I$ is $\m$-primary. The multiplicity of $I$ is defined as $$e_0(I) = \lim_{n \rt \infty} \frac{d!}{n^d}\ell(A/I^n).$$
\section{Proof of Theorem \ref{ingredient}}
In this section we give:
\begin{proof}[Proof of Theorem \ref{ingredient}]
We make several reductions to prove the result.

(1) It suffices to assume that the residue field of $A$ is uncountable.\\
If the residue field $k$ of $A$ is finite or countably infinite then we go to the flat extension $B = A[[X]]_{\m A[[X]]}$. The residue field of $B$ is $k((X))$ which is uncountable.
Set $\R^\prime = \R\otimes_A B = B[IBt]$ and $\Sc^\prime = \Sc \otimes_A B = B[JBt]$. Then note that
$\R^\prime $ is a flat extension of $\R$.  We also have $\Sc^\prime = \Sc \otimes_\R \R^\prime$. It follows from \cite[4.3.2]{bs} that for all $i \geq 0$
$$ H^i_{\R^\prime_+}(\Sc^\prime) = H^i_{\R_+}(\Sc)\otimes_{\R}\R^\prime = H^i_{\R_+}(\Sc)\otimes_A B.$$
Therefore $ H^i_{\R^\prime_+}(\Sc^\prime)_n = 0$ for all $n \gg 0$ and for all $i \geq 0$.
As we are assuming that the result holds for $B$ we get  $JB \subseteq \ov{IB}$. As $B$ is faithfully flat over $A$ we have
$$J = JB \cap A \subseteq \ov{IB} \cap A = \ov{I}.$$
(for the first equality see  \cite[7.5]{M} and for the second see \cite[1.6.2]{HS}).

(2) We assert that $J^n \cap H^0_I(A) = 0$ for all $n \gg 0$.\\
It follows easily from the Artin-Rees lemma that $H^0_I(A) \cap I^n = 0$ for all $n \gg 0$ (say from $n \geq r$).
As $\ell(J/I) < \infty$ we have $\m^sJ \subseteq I$ for some $s \geq 1$. Then $J^{s+1} \subseteq I$. It follows that $J^{r(s+1)} \cap H^0_I(A) = 0$.

(3) We may assume that the analytic spread of $I$ is positive. \\
Suppose $I$ has analytic spread zero. Then $I$ is nilpotent. So $H^0_I(A) = A$. By (2) it follows that $J$ is also nilpotent. So trivially we have $J \subseteq \ov{I}$.

(4) We may assume that $\grade I > 0$. \\
By (3) we may assume that the analytic spread of $I$ is positive. Suppose $\grade I  =0$. Set $B = A/H^0_I(A)$, $\R^\prime = B[It]$ and $\Sc^\prime = B[Jt]$. Consider the natural exact sequence of $\R$-modules
$0 \rt K \rt \Sc \rt \Sc^\prime \rt 0$. Then $K_n = J^n \cap H^0_I(A) = 0$ for $n \gg 0$. Thus $K$ is $\R_+$-torsion. So $H^0_{\R_+}(K) = K$ and $H^i_{\R_+}(K) = 0$ for $i > 0$.
 It follows that we have an exact sequence
$$ 0 \rt K \rt H^0_{\R_+}(\Sc) \rt H^0_{\R_+}(\Sc^\prime) \rt 0 \quad \text{and} \quad H^i_{\R_+}(\Sc) = H^i_{\R_+}(\Sc^\prime) \ \text{for} \ i > 0.$$
It follows that $H^i_{\R_+}(\Sc^\prime)_n = 0$ for all $n \gg 0$ and for all $i \geq 0$.
By graded independence theorem of local cohomology \cite[4.2.1]{bs}  we have
$$H^i_{\R_+}(\Sc^\prime) \cong H^i_{\R^\prime_+}(\Sc^\prime) \quad \text{for all} \ i \geq 0. $$
So $H^i_{\R^\prime_+}(\Sc^\prime)_n = 0$ for all $n \gg 0$ and for all $i \geq 0$. As $\grade IB > 0$ and we are assuming the result holds when $\grade IB > 0$; we have
$JB \subseteq \ov{IB}$. By \cite[1.2.5]{HS}  we have that $IB$ is a reduction of $JB$. Say $J^{r+1}B = I J^r B$. By (2), say $J^n \cap H^0_I(A) = 0$ for $n \geq s.$
Let $n \geq \max \{ r, s \}$. We have
$$ J^{n+1} + H^0_I(B) = IJ^n + H^0_I(B). $$
Let $a \in J^{n+1}$. Then $a = b + c$ where $b\in JI^n$ and $c \in H^0_I(A)$. Then $c \in J^{n+1} \cap H^0_I(A) = 0$. So $J^{n+1} = IJ^n$. Thus $I$ is a reduction of $J$. Again by \cite[1.2.5]{HS}
we have $J \subseteq \ov{I}$.

\

After these preliminaries we prove the result by induction on $d$. By (1) we may assume the residue field of $A$ is uncountable. By (3) the result hold when $d = 0$.

We first prove the result when $d = 1$.
By (3) we may assume that $I$ has positive analytic spread. By (4) we may assume $\grade I > 0$. In particular $A$ is \CM  \ and $I$ is $\m$-primary.
We have an exact sequence of $\R$-modules $0 \rt \Sc \rt A[t] \rt \Lc^J(-1) \rt 0$. As $\grade I > 0$ we have $H^0_{\R_+}(A[t]) = 0$. We have an inclusion
$H^0_{\R_+}(\Lc^J) \subseteq H^1_{\R_+}(\Sc)$. So $H^0_{\R_+}(\Lc)_n = 0$ for $n \gg 0$. Thus $\Lc^J$ is a quasi-finite $\R$-module.  Also as again $\grade I > 0$ we have that $\Lc^I$ is
a quasi-finite $\R$-module.  As the residue field of $A$ is uncountable there exists $x \in I$ such that $xt$ is $\Lc^I \oplus \Lc^J$-filter regular, see \ref{filt}. Note $x$ is $I$-superficial. So $(x)$ is a reduction of $I$. Therefore $I^n \subseteq (x)$ for all $n \gg 0$. As $J^s \subseteq I$ for some $s$ it follows that $J^n \subseteq (x)$ for all $n \gg 0$. Set $W = \bigoplus_{n \geq 0} J^{n+1}/I^{n+1}$. We have an exact sequence $0 \rt W \rt \Lc^I \rt \Lc^J \rt 0.$ Going mod $xt$  and as $xt$ is $\Lc^I \oplus \Lc^J$-filter regular we have an exact sequence
$$ 0 \rt V \rt W/xt W \rt \Lc^I/xt \Lc^I \rt  \Lc^J/xt \Lc^J \rt 0 \quad \text{where $V$ has finite length.} $$
We note that
$$ (\Lc^J/xt \Lc^J)_n = A/(J^{n+1}, x) = A/(x) \quad \text{for all $n \gg 0$}.$$
Similarly $(\Lc^I/xt \Lc^I)_n = A/(x)$ for all $n \gg 0$. It follows that $(W/xt W)_n = 0$ for all $n \gg 0$. Note $xt$ is also $W$-filter regular. It follows that $\ell(J^n /I^n)$ is constant for all $n \gg 0$. So $e_0(I) = e_0(J)$. As $A$ is \CM \ it follows from Rees multiplicity theorem, see \cite{R} (also see \cite[11.3.1]{HS}). that $I$ is a reduction of $J$. So $J \subseteq \ov{I}$.

Now assume $d \geq 2$ and the result has been proved for all rings of dimension $\leq d - 1$. By (3) we may assume that $I$ has positive analytic spread. By (4) we may assume $\grade I > 0$.
We have an exact sequence of $\R$-modules $0 \rt \Sc \rt A[t] \rt \Lc^J(-1) \rt 0$. As $\grade I > 0$ we have $H^0_{\R_+}(A[t]) = 0$. We have an inclusion
$H^0_{\R_+}(\Lc^J) \subseteq H^1_{\R_+}(\Sc)$. So $H^0_{\R_+}(\Lc^J)_n = 0$ for $n \gg 0$. So $\Lc^J$ is a quasi-finite $\R$-module.  Also as  $\grade I > 0$ we have that $\Lc^I$ is
a quasi-finite $\R$-module.  As the residue field of $A$ is uncountable there exists $x \in I$ such that $xt$ is $\Lc^I \oplus \Lc^J$-filter regular, see \ref{filt}. Note $(J^{n+1} \colon x) = J^n$ for all $n \gg 0.$ Similarly $(I^{n+1} \colon x) = I^n$ for all $n \gg 0$. In particular $x$ is $A$-regular. Set $B = A/(x)$, $\R^\prime = B[It]$ and $\Sc^\prime = B[Jt]$. We have an exact sequence of $\R$-modules
$$ 0 \rt \Sc(-1) \xrightarrow{xt} \Sc \rt \ov{\Sc} \rt 0. $$
It follows that $H^i_{\R_+}(\ov{\Sc})_n = 0$ for all $n \gg 0$ and for all $i \geq 0$. We have an exact sequence
of $\R$-modules
$$0 \rt K \rt \ov{\Sc} \rt \Sc^\prime \rt 0 \quad \text{for $n \geq 1$ we have} \ K_n = \frac{J^n \cap  (x)}{xJ^{n-1}}. $$
Note $K_n = 0$ for $n \gg 0$. So $K$ is $\R$-torsion. It follows that $H^i_{\R_+}(\Sc^\prime)_n = 0$ for $n \gg 0$ and for all $i \geq 0$.
By graded independence theorem of local cohomology,  \cite[4,2,1]{bs}, we have
$$H^i_{\R_+}(\Sc^\prime) \cong H^i_{\R^\prime_+}(\Sc^\prime) \quad \text{for all} \ i \geq 0. $$
So $H^i_{\R^\prime_+}(\Sc^\prime)_n = 0$ for all $n \gg 0$ and for all $i \geq 0$.
As $\dim B = d -1$, by our induction hypotheses we have $JB \subseteq \ov{IB}$. So $IB$ is a reduction of $JB$. Thus there exists $m \geq 1$ such that $J^{m+1}B = IJ^mB$.
  We also have $(J^{n+1} \colon x) = J^n$ say for $n \geq r$. Let $n \geq \max\{ m, r+1 \}$.
Then $J^{n+1} + (x) = IJ^n + (x)$. Let $a \in J^{n+1}$. Then $a = b + x\theta$ where $b \in IJ^n$. Note $x\theta \in J^{n+1}$. So $\theta \in J^n$. It follows that $a \in JI^n$.
Thus $J^{n+1} = IJ^n$. So $I$ is a reduction of $J$. It follows that $J \subseteq \ov{I}$.
\end{proof}
\section{Proof of Lemma \ref{blow-up}}
In this section we give
a proof of Lemma \ref{blow-up}. The case when $\dim A = 0$ has to be proved separately.
\begin{proof}[Proof of Lemma \ref{blow-up} when $\dim A = 0$]
We note that as $I$ is nilpotent we get that $\R_+$ is nilpotent. So $ H^0_{\R_+}(\Sc) = \Sc$. The result follows.
\end{proof}
When $d \geq 1$ we need the following result.
\begin{proposition}\label{pos}(with hypotheses as in \ref{blow-up})
Assume $d \geq 1$. If $H^0_{\R_+}(\Sc)_n = 0$ for all $n \gg 0$ then $\grade I > 0$.
\end{proposition}
\begin{proof}
  Suppose if possible $\grade I = 0$. Then $U = H^0_I(A) \neq 0$. There exists $n > 0$ such that $I^nU = 0$.
  It follows that
  $$ \bigoplus_{n \geq 0}Ut^n \subseteq H^0_{\R_+}(\Sc). $$
  This contradicts the fact that $H^0_{\R_+}(\Sc)_n = 0$ for $n \gg 0$.
\end{proof}
We use induction on $d \geq 1$ to prove Lemma \ref{blow-up}. The following result helps in the base case $d = 1$.
\begin{proposition}[with hypotheses as in Lemma \ref{blow-up}]\label{dim1}
Assume further $(A,\m)$ is  a one dimensional \CM \ local ring. Then $H^1_{\R_+}(\Sc)_n \neq 0$ for infinitely many $n \geq 0$.
\end{proposition}
\begin{proof}
  Suppose if possible $H^1_{\R_+}(\Sc)_n = 0$ for all $n \gg 0$. We have an exact sequence
  $$ 0 \rt \R \rt \Sc \rt \Lc^I(-1) \rt 0.$$
  So $H^1_{\R_+}(\Lc^I)_n = 0$ for all $\n \gg 0$. Let $\M$ be the maximal homogeneous ideal of $\R$. Then by \cite[3.1]{P2}  we have $H^i_{\R_+}(\Lc) \cong H^i_\M(\Lc)$ for all $i \geq 0$. So
  $H^1_\M(\Lc)_n = 0$ for all say for $n \geq r$.
  Set $G = G_I(A)$ the associated graded ring of $A$ with respect to $I$.
  We have an exact sequence $ 0 \rt G \rt \Lc^I \rt \Lc^I(-1) \rt 0$ of $\R$-modules. This yields a long exact sequence in cohomology
  \[
  H^0_\M(\Lc)_{n-1} \rt H^1_\M(G)_n \rt H^1_\M(\Lc^I)_n \rt H^1_\M(\Lc^I)_{n-1} \rt H^2_\M(G)_n = 0.
  \]
  So we have surjections $H^1_\M(\Lc^I)_n \rt H^1_\M(\Lc^I)_{n-1}$ for all $n \in \Z$. As $H^1_\M(\Lc)_n = 0$ for all $n \geq r$ it follows that infact  $H^1_\M(L)_n = 0$ for all $n \in \Z$.
  As $H^0_\M(\Lc)_n = 0$ for $n < 0$ it follows that $H^1_\M(G)_n = 0$ for $n \leq 0$. This is a contradiction, see \cite[17.1.10]{bs}.
  Thus $H^1_{\R_+}(\Sc)_n \neq 0$ for infinitely many $n \geq 0$.
\end{proof}
We now give
\begin{proof}[Proof of Lemma \ref{blow-up} when $\dim A > 0$]
We prove the result by induction on $d = \dim A$. We first consider the case when $d = 1$.  If $H^0_{\R_+}(\Sc)_n = 0$ for $n \gg 0$ then by \ref{pos} we have $\grade I > 0$. Thus $A$ is \CM. By \ref{dim1} we get $H^1_{\R_+}(\Sc)_n \neq 0$ for infinitely many $n \geq 0$. The result follows in this case.

Now assume $d \geq 2$ and assume the result holds for local rings of dimension $d -1$.
Assume if possible $H^i_{\R_+}(\Sc)_n = 0$ for all $n \gg 0$ and for all $i \geq 0$.
By \ref{pos} we have $\grade I > 0$. Let $x \in I$ be $A$-regular. Then $xt$ is $\Sc$-regular.
We have an exact sequence of $\R$-modules $$0 \rt \Sc(-1) \xrightarrow{xt} \Sc \rt \Sc/xt\Sc \rt 0.$$
It follows that $H^i_{\R_+}(\Sc/xt\Sc)_n = 0$ for all $n \gg 0$ and for all $i \geq 0$.
Set $B = A/(x), \R^\prime = B[IBt]$ and $\Sc^\prime = B[t]$. We have an exact sequence of $\R$-modules
$$ 0 \rt K \rt \Sc/xt \Sc \rt \Sc^\prime \rt 0 \quad \text{where } \ K = (x)t^0. $$
It follows that $K$ is $\R_+$-torsion. So  $H^i_{\R_+}(\Sc^\prime)_n = 0$ for all $n \gg 0$ and for all $i \geq 0$.
By graded independence theorem of local cohomology we have
$$H^i_{\R^\prime_+}(\Sc^\prime) \cong  H^i_{\R_+}(\Sc^\prime) \quad \text{ for all $i \geq 0$.} $$
So we have $H^i_{\R^\prime_+}(\Sc^\prime)_n = 0$ for all $n \gg 0$ and for all $i \geq 0$. This contradicts our induction hypothesis. The result follows.
\end{proof}

\section{Proof of Theorem \ref{main}}
In this section we give
\begin{proof}[Proof of Theorem \ref{main}]
By an exercise problem in \cite[4.5.12]{BH},  some Veronese of $\Sc$ is standard graded. Say $\Sc^{<l>} = A[Jt]$. Local cohomology commutes with the Veronese functor \cite[12.4.6]{bs}.  So we have
$$H^i_{\R^{<l>}_+}(\Sc^{<l>})_n = H^i_{\R_+}(\Sc)_{nl} = 0 \quad \text{for all $n \gg 0$ and for all $i \geq 0$}.$$
Furthermore $\R^{<l>} = A[I^lt]$. If we prove that $A[Jt]$ is a finite $A[I^l t]$ module (equivalently $J \subseteq \ov{I^l}$) then as $\Sc$ is a finite $A[Jt]$-module, it follows that $\Sc$ is a finite $\R$-module.  Thus it suffices to assume $\Sc = A[Jt]$ for some ideal $J$ in $A$.

Suppose if possible $J \nsubseteq \ov{I}$. Then $K = J \cap \ov{I}$ is a proper subset of $J$ containing $I$. Set $\T = A[Kt]$.  As $I$ is a reduction of $K$ we get $\sqrt{\R_+ \T} = \sqrt{\T_+}$.
It follows that $$H^i_{\R_+}(\Sc) = H^i_{\R_+ \Sc}(\Sc) = H^i_{\T_+}(\Sc) \quad \text{for all}\ i \geq 0. $$
So we may replace $K$ by $I$ and assume $\ov{I} \cap J = I$.

Let $P$ be a minimal prime of $J/I$. Then note that if $J_P \subseteq \ov{I_P}$ then
$$ J_P = (\ov{I_P})\cap J_P = (\ov{I})_P \cap J_P = I_P \quad \text{a contradiction.}$$
Note $J_P/I_P$ is a non-zero and of finite length. Furthermore $J_P \nsubseteq \ov{I_P}$. Set $B = A_P$, $\R^\prime = B[IBt]$ and $\Sc^\prime = B[Jt]$. Then we have
\begin{equation*}
  H^i_{\R^\prime_+}(\Sc^\prime)_n = 0 \quad \text{for all $n \gg 0$ and for all $i \geq 0$.} \tag{$\dagger$}
\end{equation*}

We have to consider two cases:\\
Case (1):  $J_P = B$. Then as $B_P/I_P$ has finite length we have $I_P$ is $PB$-primary. We also have $\Sc^\prime = B[t]$.
Then by Lemma \ref{blow-up}, $(\dagger)$ is NOT possible.

Case (2) $J_P$ is a proper ideal of $B$. Then as $\ell(J_P/I_P)$ is finite the condition $(\dagger)$ implies by Theorem \ref{ingredient} that $J_P \subseteq \ov{I_P}$, a contradiction.

Thus our assumption that $J  \nsubseteq \ov{I}$ is not possible. So $J \subseteq \ov{I}$.
\end{proof}

\end{document}